\documentclass[12pt]{amsart}
\usepackage{amssymb, amsthm, amscd, epsfig, graphs, psfrag, graphicx,
enumerate, stmaryrd, amsmath, booktabs, hhline, xcolor, mathtools, float} 

\newtheorem{thm}{Theorem}[section]
\newtheorem{lem}[thm]{Lemma}
\newtheorem{prop}[thm]{Proposition}

\theoremstyle{definition}
\newtheorem{cor}[thm]{Corollary}
\newtheorem{defn}[thm]{Definition}
\newtheorem*{alg}{Algorithm}

\theoremstyle{remark}
\newtheorem{rem}[thm]{Remark}

\numberwithin{equation}{section}

\newcommand{\De}{\Delta}
\newcommand{\de}{\delta}

\newcommand{\si}{\sigma}

\newcommand{\va}{\varphi}

\newcommand{\x}{\times}

\newcommand{\Z}{\mathbb Z}

\newcommand{\R}{\mathbb R}

\newcommand{\del}{\partial}

\newcommand{\co}{\colon\thinspace}

\begin{document}
\mathsurround=1pt 
\title{Morse functions constructed by random walks}




\author{Boldizs\'ar Kalm\'{a}r}



\address{E\"{o}tv\"{o}s Lor\'{a}nd University, P\'{a}zm\'{a}ny P\'{e}ter s\'{e}t\'{a}ny 1/c., 1117 Budapest, Hungary}

\address{Budapest University of Technology and Economics, Institute of Mathematics, Egry József street 1., 1111 Budapest, Hungary}
\email{boldizsar.kalmar@gmail.com}

\begin{abstract}
We construct  random Morse functions on surfaces by random walk and compute related distributions. 
 We study the space of Morse functions 
 through these random variables.
  We consider subspaces characterized by 
 the surfaces with boundary obtained by cutting the closed domain surface of the Morse function 
 at the levels of  regular values.
We consider Morse functions having a bounded number of
 critical points and one single local minimum. 
We find a small set of Morse functions which are close enough to any other Morse function 
in the sense that they share the same characterizing surfaces with boundary.
\end{abstract}

\maketitle


\section{Introduction}

A Morse function on an $n$-dimensional  closed manifold $M$ is a smooth map
$f \co M \to \R$ which has the  local form 
$
(x_1, \ldots, x_n ) \mapsto \sum_{i=1}^n \pm x_i^2
$
around each singular point. 
Morse functions are important in studying differentiable manifolds, many constructions for them are related to topological properties of smooth manifolds. 
 In the present paper, we deal with the space of Morse functions on closed orientable surfaces.
 We define some random variables of Morse functions 
 indicating  {the number of boundary circle components and  the
  genus} of 
the surface 
  {with boundary obtained by cutting the surface along a level
of a regular value (of the Morse function).}
 By genus we mean the genus of the closed surface that we get by attaching
disks to the boundary circle  components. 
 We apply some approach and statements of the probabilistic method \cite{AS00} to 
 get results about the subspaces of Morse functions defined 
  {from this perspective}.
 We consider two Morse functions $f_{1, 2} \co M \to \R$ to be equivalent 
  if there are  some diffeomorphism $\va \co M \to M$  {and 
  orientation preserving diffeomorphism $\psi \co \R \to \R$} such that 
 {$f_2 = \psi \circ f_1 \circ \va$}.   {However, 
 our process can yield many non-equivalent constructed Morse functions 
 at the same time on a single surface (depending on the arrangement of critical values in $\R$), we explain this in detail  later.} 

   Random constructions for mathematical structures 
are often used when we want to show the existence of some structure with a special property, in combinatorics this method was initiated by T.\ Szele and P.\ Erd\H{o}s in the '40s. For example, random subgraphs can be used to prove many statements \cite{AS00} in graph theory, random 
$3$-manifolds were studied in \cite{DT06a, DT06b}, random groups in \cite{Gr00, Gr03} and random knots in \cite{EHLN16}.  
 In this paper, we construct random Morse functions on surfaces by  {using} random walk and compute some related probabilities and distributions.
We consider Morse functions with one single local minimum. 
During the random walk process, we can step in the set $[0, \infty)^2  \cap \Z^2$ in three  {directions: right, left and up-and-left 
 with 
 probabilities $p_r, p_l, p_d > 0$, respectively, with $p_r + p_l + p_d = 1$ all along starting at $(1, 0)$. Each step corresponds to attaching prescribed pieces of Morse functions 
with one critical point (an upside-down pair of pants, a local maximum and a usual pair of pants, respectively)
 to get finally Morse functions} on a closed surface when the random walk reaches some $(0, z)$ for the first time.   
At the attaching of the Morse function pieces we do not specify exactly which boundary circle we use so in this sense our model does not describe completely 
the construction of Morse functions as we imagine usually when we sketch one. In other words, the constructed Morse function for one specified walk is not unique.
   {So in fact we construct sets of Morse functions having the order (according to the orientation of the target $\R$) of specific indices of critical values  in common.
   We do not study the number of Morse functions in these sets in the present paper.} 
 
 The random walk would be  a random variable of  random Morse functions and we compute properties of this random variable.
If $\mathcal O$ is the $\si$-algebra of walks on the set $\Omega'$ of walks and on the set $\Omega$ of Morse functions with one local minimum we have the map
$W \co \Omega \to \Omega'$ assigning to a Morse function the corresponding walks by the above indicated procedure, then 
the $\si$-algebra $\{ W^{-1}(A) : A \in \mathcal O \}$ yields a probability space on Morse functions with the
 measure $P(W^{-1}(A) ) = P(A)$  {if $P(W(\Omega))=1$}.  {Note that the condition $p_l + p_d \geq p_r$ implies 
 that the walk will arrive back to some $(0, z)$ with probability $1$.} So 
 the map $W \co \Omega \to \Omega'$ is a random variable  {(that is a measurable map)} of random Morse functions. 
  A random variable of walks is a variable of Morse functions too, we compute some distributions and expected values 
 of them in Proposition~\ref{variables}. 
 
 For example,  {in Proposition~\ref{variables}}
  we obtain that 
  the expected number of the critical points of a random Morse function is $1+\frac{1}{p_l+p_d - p_r}$
if $p_l + p_d > p_r$.   
In Proposition~\ref{exp_genus}
 we prove that 
if $p_l + p_d > p_r$, then the expected genus of the domain surface of a random Morse function with one local minimum is equal to 
$$\frac{  p_d + (p_l - p_d)( p_l + p_d - p_r)}{2(p_l + p_d)(p_l + p_d - p_r)}.$$ 
To get large enough values by these formulas we have to choose the probabilities
$p_r, p_l, p_d$ so that $p_l + p_d - p_r$ is small enough. For example 
with $p_d = 1/2, p_l = 1/20, p_r = 9/20$ we have that an average Morse function with one local minimum
has $11$ critical points and the genus of its domain surface is around $4.136$.
 For example, the cobordism class in $\Z$ of a Morse function, see \cite{IS03}, is then expectedly equal to 
 approximately 
 $0.364$  not only by Corollary~\ref{cob_exp} but also because
  by an easy summation we have that 
  the number of critical points minus two 
   is always equal to twice the sum of the cobordism class and the genus (if the Morse function has only one local minimum).  
   
   We would like to find an optimal set $\widetilde {\mathcal D}$ of Morse functions on a given 
 oriented surface of  genus $g \geq 0$ having the property that every Morse function 
 is ``close'' to $\widetilde {\mathcal D}$ in a  sense. We will do this on the level of random walks instead of Morse functions and we are looking for 
  such an optimal set of walks  $\mathcal D$. 
  We define a graph of random walks $\mathcal G$ such that the $i$-th  vertex corresponds to a walk $w_i$ and so corresponds to many Morse functions each of which 
  is mapped to the walk $w_i$ by the random variable $W$.  
  The optimal set of walks $\mathcal D$ is nothing else just a dominating set (a set of vertices being connected 
  to all of the vertices) 
   of the graph $\mathcal G$. 
  The optimal set $\widetilde{ \mathcal D}$ can be obtained   then 
 from  the $W$-preimage of  $\mathcal D$.
   The graph 
 $\mathcal G$ is like a string graph of the walks or
the vertex intersection graph of the walks for Morse functions (vertex intersection graphs were introduced by \cite{ACGLLS12}). 
We obtain in 
 {Corollary~\ref{maincor}} 
that 
  in the set of  Morse functions with at most $N \geq 2g + 2$ critical points and one local minimum on a genus $g \geq 2$ surface 
there is such a $\widetilde {\mathcal D}$
 consisting of at most 
 $$
 M_{N-2} \left( 1 - \frac{\De}{(\De+1)^{1+\frac{1}{\De}}}  \right)
 $$
 Morse functions, where 
 $\De+1$ is equal to  the $g$-th Catalan number and 
 $M_{N-2}$ is equal to 
$
 \sum_{k = g}^{(N-2)/2} 
\frac{1}{2k+1} \binom{2k+1}{k+1}\binom{k}{g}
$. 
 {This follows from our Theorem~\ref{main}, where we apply  \cite[Chapter I, Theorem 1.2.2]{AS00}}. 

The ``closeness'' of a Morse function to the set $\widetilde {\mathcal D}$
means that 
cutting an arbitrary Morse function $f \notin \widetilde {\mathcal D}$ at the level set $f^{-1}(a)$ of some appropriate  regular value $a \in \R$ 
the obtained surface $f^{-1}((-\infty, a])$ has 
\begin{enumerate}
\item
genus $G$, that is if we close the surface by disks it has genus $G$, and 
\item
$C$ copies of boundary circles
\end{enumerate}
 such that the pair $(C, G)$ can be obtained by similarly cutting some Morse function contained in  $\widetilde{ \mathcal D}$.
   The cases of $1 \leq G \leq g-1$ are more interesting for us than the cases  of $G = 0$ and $G = g$.

 The paper is organized as follows. In Section 2 we give our basic definitions and in Section 3 we prove our main results.

\subsection*{Acknowledgement}
 The author thanks the referee for the comments which improved the paper.

\section{Preliminaries} 

On the plane $\R^2$ let us consider 
the set 
$$
\mathcal S = \{ (x, y) : x, y \in \Z \mbox{\ and\ } x \geq 1, y \geq 0 \},$$
where we start a random walk at the point $S_0 = (1, 0)$,
 we can step to the right, to the left and to up-and-left with positive probabilities $p_r$, $p_l$ and $p_d$, respectively, where
$p_r + p_l + p_d = 1$.  
 {We imagine that during these steps a surface, and its
corresponding height functions (Morse functions), are built in the
following way.}
 \begin{alg}\label{length}
 We stay 
 in the set $\mathcal S$ all along at the choices of random steps.  {The $n$-th step is denoted by $S_n$. In each step, we 
  modify possibly many Morse functions simultaneously.} 
   {If at some step $S_n = (1, y)$, $n \geq 0$, we step outside
of $\mathcal S$ (which can only happen at a left or diagonal step)} 
 for the first time, then we consider our Morse functions to be constructed by attaching one local maximum to each of them and we finish the process. 
 Then we say that the random walk had $n$ steps.
 Then the constructed Morse functions will have $n+2$ critical points.
 The  Morse functions are built as follows. 
 At the starting point we have one local minimum mapped to $\R$ as usual and at 
 each step we raise our Morse functions:
 \begin{enumerate}
 \item
  at a step to the right 
 we raise our Morse functions by one pair of pants (the pants are upside down) attached to 
 a circle boundary component, 
 \item
 at a step to the left we attach one local maximum to some circle boundary component
  and 
  \item
  at a step to up-and-left 
  we attach a pair of pants (in the position as pants usually are) to two circle boundary components. 
 \end{enumerate} 
 \end{alg} 

 An example can be seen on Figure~\ref{example}. 
 
 It is easy to see that  Morse functions on a connected closed orientable surface 
 are constructed when the random walk arrives 
 to a point of the form $(1, y)$  after $n$ steps, where $n \geq 0$, and we add one local maximum to the single boundary circle of the surface. 
   {This surface is unique by Morse homology and the classification of closed orientable surfaces.} 
  {By simply
counting the boundary components after each step, we see that a
left or diagonal move reduces the boundary components by $1$ and a
right move increases them by $1$, so $S_n = (x, y)$ exactly if  we have $x$ circle
boundaries and the surface with boundary has genus $y$ (by a simple
argument using Morse homology).}  
  {Of course,} in this model the exact locus of the attaching
  at each step
  is undefined if there are more boundary components than attaching circles.

\begin{figure}[h!]
\begin{center}
\epsfig{file=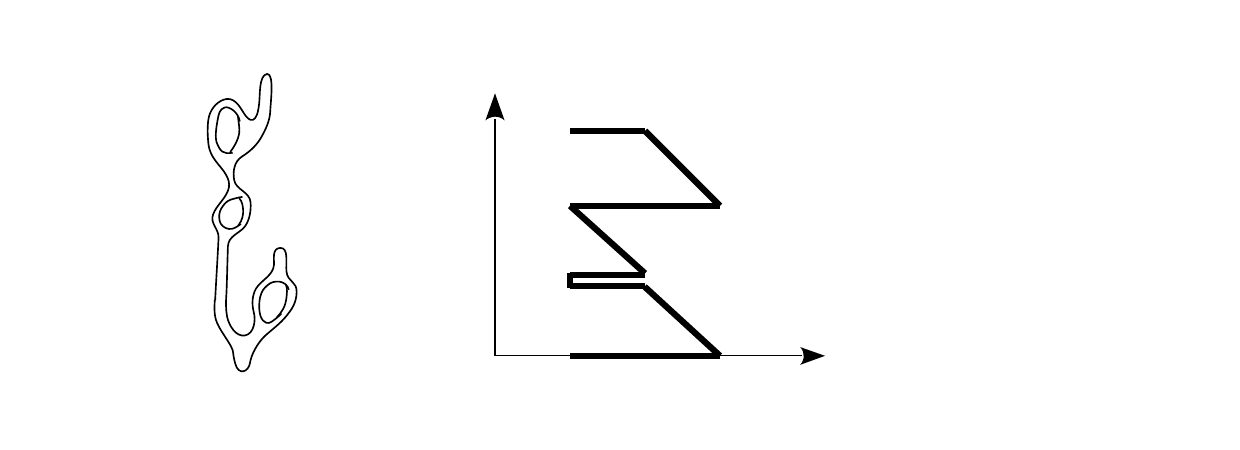, height=7cm}
\put(-10.9, 1.7){$(1, 0)$}
\put(-10.9, 5.2){$(1, 3)$}
\put(-8.2, 1.7){$(3, 0)$}
\end{center} 
\caption{A Morse function as a height function on the left and the walk corresponding to it on the right. The length of the walk is equal to $10$, 
the number of critical points is equal to $12$. The genus of the domain surface is equal to $3$. The walk starts at $(1, 0)$ and stops at $(1, 3)$. 
The two bottom indefinite critical points correspond to the first two steps to the right starting at $(1,0)$, the third indefinite critical point (that is the first pair of pants in usual position resulting 
a twice punctured torus) corresponds to the step $(3, 0) \to (2,1)$, etc. }
\label{example}
\end{figure}

  Obviously the number of critical points of the Morse function under construction after the $n$-th step in $\mathcal S$ is equal to 
$n+1$, where $n \geq 0$, and finally we attach a local maximum when we step out from $\mathcal S$. So at the end of the process
 the Morse function on the closed surface has $n+2$ critical points if we had $S_n = (1, y)$ and then we left $\mathcal S$.   
If for some $n \geq 0$ we have $S_n = (x, y)$, then 
the number of index $1$ critical points 
with usually positioned (not upside down) pair of pants 
is equal to $y$. Also then the number 
of right steps minus the number of left steps minus the number of up-and-left steps 
is equal to $x-1$.

\begin{lem}\label{step_number}
If $S_n = (x, y)$, then 
\begin{enumerate}[\rm (a)]
\item
the number $r$ of right steps is equal to 
$(n+x-1)/2$,
\item
the number $l$ of left steps is equal to 
$(n-x+1)/2 - y$, 
\item
the number $d$ of up-left steps is equal to $y$.
\end{enumerate}
So $S_n = (1, y)$ after  $n \geq 0$ steps if and only if  
$r = n/2$, 
$l = n/2 - y$ and $d = y$. 
\end{lem}
\begin{proof}
Of course $r+l+d = n$, $r-l-d = x-1$ and $d=y$ give the result.
\end{proof}
 
 We want to find an optimal set $\widetilde {\mathcal D}$ of Morse functions on a given  {closed} 
 oriented surface of  genus $g \geq 0$ having the property that any Morse function 
 is close to $\widetilde {\mathcal D}$ in some sense. We are going to do this on the level of random walks instead of Morse functions and we are looking for 
  such an optimal set of walks  $\mathcal D$. 
   {Recall from Introduction that we have the measurable map 
$W \co \Omega \to \Omega'$ assigning to a Morse function the corresponding walks by the above indicated procedure.} 
  We define a graph of random walks $\mathcal G$ such that every vertex corresponds to a walk $w_i$ and so corresponds to many Morse functions each of which 
  is mapped to the walk $w_i$ by the random variable $W$.  
  The optimal set $\widetilde{ \mathcal D}$ would be  then 
  the $W$-preimage of our optimal set of walks $\mathcal D$. We also put some restrictions on the genus $g$ and the maximal number $N$  of 
  critical points of Morse functions as follows. 
  {We define this optimal set of Morse functions with the help of constructing a graph.}

\begin{defn}
For $N \geq 2$ and $g \geq 0$ let $\widetilde{ \mathcal G}_{N, g}$ be the graph
\begin{enumerate}
\item
whose vertices 
are the Morse functions with  at most $N$ critical points on a closed surface of genus $g$  and
\item
whose edges 
are exactly between those vertices 
which 
are Morse functions $f_1$ and $f_2$ with the property that 
at the level sets $f_1^{-1}(a_1)$ and $f_2^{-1}(a_2)$ of some appropriate  regular values $a_1 \in \R$  for $f_1$ and 
$a_2 \in \R$  for $f_2$ 
the obtained surfaces $f^{-1}((-\infty, a_1])$ 
and $f^{-1}((-\infty, a_2])$ 
have the same number of boundary circles 
and the same genus. This common genus $G$ is required to satisfy  $1 \leq G \leq g-1$. 
\end{enumerate}
A dominating set in $\widetilde{ \mathcal G}_{N, g}$, that is a subset $\widetilde{ \mathcal D}_{N, g}$ of the vertices of $\widetilde{ \mathcal G}_{N, g}$ 
 such that any vertex of $\widetilde{ \mathcal G}_{N, g}$ is in $\widetilde{ \mathcal D}_{N, g}$ or has a neighbor in $\widetilde{ \mathcal D}_{N, g}$,
 is called a   \emph{dominating set} of Morse functions.
\end{defn}

 \begin{rem}
  {It is possible to describe the dominating sets of Morse functions more directly.}
 Take a  set $\widetilde{ \mathcal D}_{N, g}$ of Morse functions 
 satisfying the condition that 
cutting an arbitrary Morse function $f \notin \widetilde{ \mathcal D}_{N, g}$ at the level set $f^{-1}(a)$ of some appropriate  regular value $a \in \R$ 
the obtained surface $f^{-1}((-\infty, a])$ has 
\begin{enumerate}
\item
genus $G$, that is if we close the surface by disks it has genus $G$, and 
\item
$C$ copies of boundary circles
\end{enumerate}
 such that the pair $(C, G)$ can be obtained by similarly cutting some Morse function contained in  $\widetilde{ \mathcal D}_{N, g}$.
   We require $1 \leq G \leq g-1$. 
  For all Morse functions we  put the restriction that the number of their 
  critical points has to be at most $N \geq 2$.
  Such a set $\widetilde{ \mathcal D}_{N, g}$ of Morse functions  on a genus $g \geq 0$ 
domain surface is  a {dominating set} of Morse functions. 

A dominating set in $\widetilde{ \mathcal G}_{N, g}$ corresponds clearly to 
a dominating set $\widetilde{ \mathcal D}_{N, g}$ of Morse functions. 
\end{rem}

We require $1 \leq G \leq g-1$ 
because
   the local minimum or the index one critical points for the lowest function values can be 
    easily common at two different Morse functions, similarly near the local maxima we do not see too much interesting differences 
    between  two Morse functions on the same genus $g$ surface. 
    
\begin{rem}\label{kicsi_dom}
    Note that if $f \in \widetilde{ \mathcal D}_{N, g}$ where 
    $\widetilde{ \mathcal D}_{N, g}$ is dominating, then the set 
    $$
    \widetilde{ \mathcal D}_{N, g} - \{ g \in W^{-1} ( W ( f )) : g \neq f \}$$
is also dominating with less Morse functions.  
\end{rem}


We define the notions analogous to the graph $\widetilde{ \mathcal G}_{N, g}$ and
 the dominating sets $\widetilde{ \mathcal D}_{N, g}$ for the random walks too. 
 We define the graph ${\mathcal G}_{N-2, g}$ as follows. 

\begin{defn}
Let $N \geq 2$ and $g \geq 0$. 
The vertices of the graph ${\mathcal G}_{N-2, g}$ are in bijection with the 
random walks of length at most $N-2$ in $\mathcal S$ starting from $(1, 0)$ and finishing at the final point $(1, g)$. 
 The edges of ${\mathcal G}_{N-2, g}$ are exactly between those vertices which correspond to 
 walks $w_1$ and $w_2$ such that $w_1$ and $w_2$ 
 intersect each other in a point $(C, G) \in \mathcal S$, where $1 \leq G \leq g-1$.
\end{defn}

In fact, we are looking for dominating sets ${\mathcal D}$  in  the graph ${\mathcal G}_{N-2, g}$. 
Then $W^{-1}({\mathcal D})$ is a dominating set 
$\widetilde{ \mathcal D}_{N, g}$ in $\widetilde{ \mathcal G}_{N, g}$.
 The graph 
 $\mathcal G_{N-2, g}$ is almost like a string graph of the walks and also 
the vertex intersection graph of the walks for Morse functions in $\mathcal S$. 
 {For string or vertex intersection graphs of paths, see, for example \cite{ACGLLS12}.}

\section{Results}

Observe that if $S_n = (1, y)$, then  {obviously $n$ is even and also} $0 \leq y \leq n/2$. 
In this paper, we consider Morse functions with only one local minimum.

\begin{prop}\label{variables}\noindent
Suppose that $p_l + p_d \geq p_r$. 
\begin{enumerate}[\rm (1)]
\item
The probability that the random walk starting at $(1, 0)$ arrives to a point of the form $(0, z)$ is equal to $1$. 
\item
Let $n \geq 0$ be an even integer. 
The probability  that  {from  a} walk a  random Morse function is created with $n+2$ critical points 
on a closed surface is equal to 
$$
\frac{1}{n+1}\binom{n+1}{(n+2)/2} p_r^{n/2}(p_l+p_d)^{(n+2)/2}.
$$
\item
The expected number of the critical points of a random Morse function is $$1+\frac{1}{p_l+p_d - p_r}$$
if $p_l + p_d > p_r$ and $\infty$ if $p_l + p_d = p_r$. 
\item
Let $y \geq 0$ and $n \geq 0$ be an  even integer. 
The probability  that during the walk a  random Morse function is created with $n+2$ critical points 
on a closed surface of genus $y$ (so $n$ is necessarily even and the walk finishes at $S_{n+1} = (0,z)$, $S_n = (1, y)$ and we start at $(1, 0)$) 
is equal to 
$$
\frac{1}{n+1}\binom{n+1}{(n+2)/2}p_r^{n/2}(p_l+p_d)\binom{n/2}{y}p_l^{n/2-y}p_d^y.
$$
\item
The expected number of local maxima of a random Morse function with one local minimum is equal to 
 $$\frac{p_d (p_l + p_d - p_r) + p_l/2 }{(p_l + p_d)(p_l + p_d - p_r)}$$ if $p_l + p_d > p_r$. 
\end{enumerate}
\end{prop}
\begin{proof}
For (1) note that it is well-known that if $p_l + p_d \geq p_r$, then the probability of hitting $(0, z)$ for some $z$ is equal to $1$, since we can project our points 
in $\Z \x \Z$ to the first coordinates to get a $1$-dimensional random walk and we can apply the well-known theory,  see \cite[page 347  (2.8)]{Fe68}. 

For (2) the distribution of the first passage at $0$ of a $1$-dimensional random walk with initial position $z>0$ is well-known \cite[Chapter XIV, Section 4]{Fe68} and 
if it happens at the $(n+1)$-th step with $z=1$, then 
it is equal to 
$$
\frac{1}{n+1}\binom{n+1}{(n+2)/2}p_r^{n/2} q^{(n+2)/2},
$$
where $q = p_l + p_d$ since we project to the first coordinates our position in $\Z \x \Z$.

For (3) we can use \cite[Chapter XIV, Section 3]{Fe68}. If our walking process has length $n$ in the sense of  {our algorithm on page 3}, then the resulting 
 Morse function has $n+2$ 
critical points while the length of the random walk  is equal to $n+1$. 
So by \cite[Chapter XIV, Section 3]{Fe68}
 since the expected duration of the walking is 
 $\frac{1}{p_l+p_d - p_r}$, the expected number of critical points is equal to $1+\frac{1}{p_l+p_d - p_r}$ if $p_l+p_d > p_r$.
 If $p_l+p_d = p_r$, then the expected duration is infinite.

For (4) note that the number of walks in $\mathcal S$ is equal to 
$$
\frac{1}{n+1}\binom{n+1}{(n+2)/2} 2^{n/2}
$$
since the number of walks in $1$-dimension (after the projection to the first coordinates our position in $\Z \x \Z$) is equal to 
 $
\frac{1}{n+1}\binom{n+1}{(n+2)/2}
$ by basic combinatorics. 
Consider (2) and that in order to reach height $y$, that is genus $y$, 
we can have $y$ number of up-and-left steps in $\binom{n/2}{y}$ different ways so
the number of walks is equal to 
$$
\frac{1}{n+1}\binom{n+1}{(n+2)/2} \binom{n/2}{y}.
$$
Each walk has probability 
$p_r^{n/2}(p_l+p_d) p_l^{n/2-y}p_d^y$ so we get the result. 

%
%

For (5) the expected value is equal to 
$$
\sum_{
 n \geq 0,\thinspace  2|n
}  \frac{1}{n+1}\binom{n+1}{(n+2)/2}p_r^{n/2}(p_l+p_d) \sum_{0 \leq y \leq n/2}  \left(\frac{n}{2}-y +1 \right) \binom{n/2}{y}p_l^{\frac{n}{2}-y}p_d^y 
$$
  {by the standard way to compute the expected value} 
since 
for a constructed Morse function the number of local maxima is equal to $\frac{n}{2}-y +1$ by Lemma~\ref{step_number} 
   {considering genus $y$ surfaces}.
 We have 
\begin{multline*}
\sum_{0 \leq y \leq n/2}  \left(\frac{n}{2}-y +1 \right) \binom{n/2}{y}p_l^{\frac{n}{2}-y}p_d^y = 
 \left (  \frac{\del}{\del x}  
 \sum_{0 \leq y \leq \frac{n}{2} }  \binom{n/2}{y}x^{\frac{n}{2}-y+1}p_d^y \right)_{x = p_l}= \\ 
  \frac{\del}{\del x}  x(x + p_d)^{n/2}|_{x = p_l}  =  (p_l + p_d)^{\frac{n}{2}} +   p_l \frac{n}{2}(p_l + p_d)^{\frac{n}{2}-1} = 
  (p_d + \frac{n+2}{2}p_l )(p_l + p_d)^{\frac{n}{2}-1}
\end{multline*}
if $n/2 \geq 1$, otherwise for $n=0$ we have $1$. 
 So the expected number of
the local maxima  is equal to 
\begin{multline*}
p_l+p_d + 
\sum_{n \geq 1,\thinspace   2|n} \frac{1}{n+1}\binom{n+1}{(n+2)/2}p_r^{n/2} (p_l + p_d) 
(p_d + \frac{n+2}{2}p_l )(p_l + p_d)^{\frac{n}{2}-1} = \\
 \sum_{n \geq 0,\thinspace   2|n}^{\infty} \frac{1}{n+1}\binom{n+1}{(n+2)/2} p_r^{\frac{n}{2}} (p_l + p_d)^{\frac{n}{2}} (p_d + \frac{n+2}{2}p_l ) = \\
 p_d \sum_{n \geq 0,\thinspace   2|n}^{\infty} \frac{1}{n+1}\binom{n+1}{(n+2)/2} p_r^{\frac{n}{2}} (p_l + p_d)^{\frac{n}{2}}  + 
\frac{p_l}{2}   \sum_{n \geq 0,\thinspace   2|n}^{\infty} \frac{n+2}{n+1}\binom{n+1}{(n+2)/2} p_r^{\frac{n}{2}} (p_l + p_d)^{\frac{n}{2}} = \\
  \frac{p_d}{p_l + p_d}    +  \frac{p_l}{2}  \frac{1}{(p_l + p_d)(p_l + p_d - p_r)} = 
  \frac{p_d (p_l + p_d - p_r) + p_l/2 }{(p_l + p_d)(p_l + p_d - p_r)}.
 \end{multline*}
\end{proof}

Recall that the cobordism group of Morse functions on oriented surfaces is isomorphic to 
$\Z$ and the cobordism class of a Morse function is given by the number of local maxima minus the number of local minima \cite{IS03}. 

\begin{cor}\label{cob_exp}\noindent 
\begin{enumerate}
\item
The expected cobordism class in $\Z$  of a random Morse function with one local minimum is equal to 
$$
 \frac{p_d (p_l + p_d - p_r) + p_l/2 }{(p_l + p_d)(p_l + p_d - p_r)} - 1 = \frac{ \frac{p_l}{2}  - p_l(p_l + p_d - p_r)}{(p_l + p_d)(p_l + p_d - p_r)}.
$$
\item
If $p_l + p_d > p_r$, then the expected number of index one critical points of a random Morse function with one local minimum is equal to 
$$
1+\frac{1}{p_l+p_d - p_r} - \frac{p_d (p_l + p_d - p_r) + p_l/2 }{(p_l + p_d)(p_l + p_d - p_r)} - 1 = \\
\frac{\frac{p_l}{2} + p_d (1 - p_l - p_d + p_r)}{(p_l + p_d)(p_l + p_d - p_r)}
$$
because we get the number of index one critical points 
by subtracting from the number of all critical points the number of index $0$ or $2$ critical points. 
\end{enumerate} 
\end{cor}

\begin{prop}\label{exp_genus}
If $p_l + p_d > p_r$, then the expected genus of the domain surface of a random Morse function with one local minimum is equal to 
$$ \frac{  p_d + (p_l - p_d)( p_l + p_d - p_r)}{2(p_l + p_d)(p_l + p_d - p_r)}.$$
\end{prop}
\begin{proof}
 {By the fact that the cobordism class of a Morse function is equal to the number of index one critical points minus twice 
  the number of index one critical points in the usual pair of pants position, see  \cite{IS03}, and that 
  the genus is equal to  the number of index one critical points in the usual pair of pants position because we have always one 
  local minimum, we get}
\begin{multline*}
\frac{1}{2} \left( \frac{\frac{p_l}{2} + p_d (1 - p_l - p_d + p_r)}{(p_l + p_d)(p_l + p_d - p_r)} -
\frac{ \frac{p_l}{2}  - p_l(p_l + p_d - p_r)}{(p_l + p_d)(p_l + p_d - p_r)} \right) = \\
 \frac{  p_d (1 - p_l - p_d + p_r) + p_l(p_l + p_d - p_r)}{2(p_l + p_d)(p_l + p_d - p_r)} = 
    \frac{  p_d + (p_l - p_d)( p_l + p_d - p_r)}{2(p_l + p_d)(p_l + p_d - p_r)}.
\end{multline*}
\end{proof}

%
%

Let  $g \geq 0$ and $2g  + 2 \leq N$, where recall that $N$ is the upper bound for the number of
 critical points. 
 By the previous statements 
 the number of walks of length at most $N-2$ is equal to 
 $$
 \sum_{
\begin{array}{c}
 {\scriptstyle 2g\leq n \leq N-2  }  \\
 {\scriptstyle 2|n }
\end{array}}
\frac{1}{n+1} \binom{n+1}{\frac{n}{2}+1}\binom{n/2}{g}
$$
 and denote this number by $M_{N-2}$. 
 {Let $\De$ denote the $g$-th Catalan number minus one, that is 
$$
\De +1 =  \frac{1}{2g+1} \binom{2g+1}{g+1} = \frac{(2g)!}{(g+1)! g!}.
$$
Then of course we have 
$$
\De = M_{2g} - 1.
$$}

\begin{rem}
Let  $g \geq 2$ and $2g   \leq N - 2$. Then 
 in the graph $\mathcal G_{N-2, g}$ 
the minimum degree is at least one because
 of the following.
 If the first step during the walk from height $0$ to height $1$
 happens at $(2, 0)$, then 
 the next step is from  $(1, 1)$ to $(2, 1)$ and we can change these last 
  two steps to $(2, 0) \to (3, 0) \to (2, 1)$ to get a different walk with the same length intersecting each other at height $(2,1)$.
  If the first step during the walk from height $0$ to height $1$
 happens at some $(a, 0) \to (a-1,1)$ with $a \geq 3$, then 
 the previous step is $(a-1, 0) \to (a, 0)$ or $(a+1, 0) \to (a, 0)$ 
  and we can change the steps 
  $ (a-1, 0) \to (a, 0) \to (a-1, 1)$ or $(a+1, 0) \to (a, 0) \to ( a-1, 1)$, respectively, 
  to 
  $ (a-1, 0) \to (a-2, 1) \to (a-1, 1)$ or $(a+1, 0) \to (a, 1) \to ( a-1, 1)$, respectively. 
  Either way we get
  another walk with the same length intersecting the first walk at $(a-1,1)$. 
  So in $\mathcal G_{N-2, g}$ 
the minimum degree is at least one. 
\end{rem}

\begin{thm}\label{main}
Let  $g \geq 2$ and $2g   \leq N - 2$. Then 
there exists a dominating set $\mathcal D$
 of the graph  $\mathcal G_{N-2, g}$ 
 of at most 
 $$
 M_{N-2} \left( 1 - \frac{\De}{(\De+1)^{1+\frac{1}{\De}}}  \right)
 $$
 vertices. 
\end{thm}
\begin{proof}
By \cite[Chapter I, Theorem 1.2.2]{AS00} and the explanation after it, 
 in a graph of $n$ vertices and minimum degree at least $\de \geq 1$ there is a dominating set
of at most 
$$
n\left(1- (\de + 1)^{-1/\de}  + (\de +1)^{-(\de+1)/\de} \right) = n\left(1-  \frac{\de}{(\de+1)^{1+\frac{1}{\de}}} \right)
$$
vertices as we can see easily.
 To show the statement we have to compute the number of vertices $n$ and we have to show
 that the minimum degree in $\mathcal G_{N-2, g}$  is actually at least $\De$. 
  Note that $\De \geq 1$ because 
  $$
  \De + 1 = \frac{(2g)!}{(g+1)! g!} = \frac{2g(2g-1)\cdots (g+2)}{g!} \geq \frac{(g+2)(g+1)\cdots 4}{g!} = (g+2)(g+1)/6
 $$  
 for $g \geq 2$ so 
 $
 \De + 1 \geq  (g+2)(g+1)/6 \geq 2$. 
 As we mentioned, by the proof of Proposition~\ref{variables} (4) 
 the number of vertices $n$ is equal to $M_{N-2}$.
Also the minimum degree of $\mathcal G_{N-2, g}$ is at least
 $\De$ because every walk intersects another 
 $\De$ since
 if the last step of a walk from height $g-1$ to height $g$
 is at $(a+1, g-1) \to (a, g)$, where $a \geq 1$, then 
 all the possible simplest walks from $(a, 0)$ to $(a, g)$
 and an additional walk from $(1, 0)$ to $(a, 0)$ 
 give at least $\De +1$ walks (and our original walk is maybe among them). 
  By ``simplest'' we mean 
  the walks
  that arise when the walking length 
  is equal to $2g$, see Figure~\ref{simplestwalks}. 
 \end{proof}

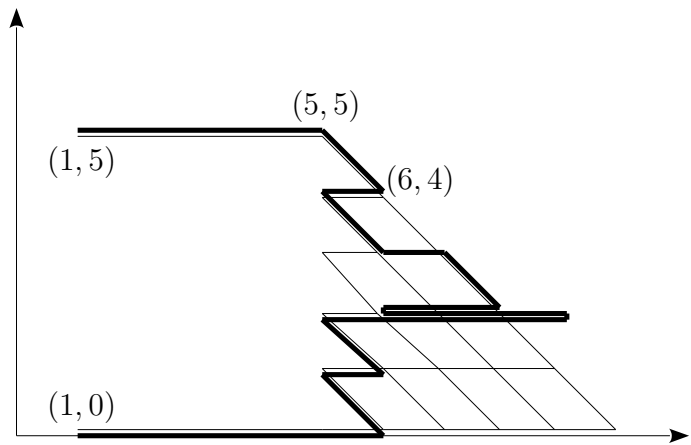
\begin{figure}[h!]
\begin{center}
\epsfig{file=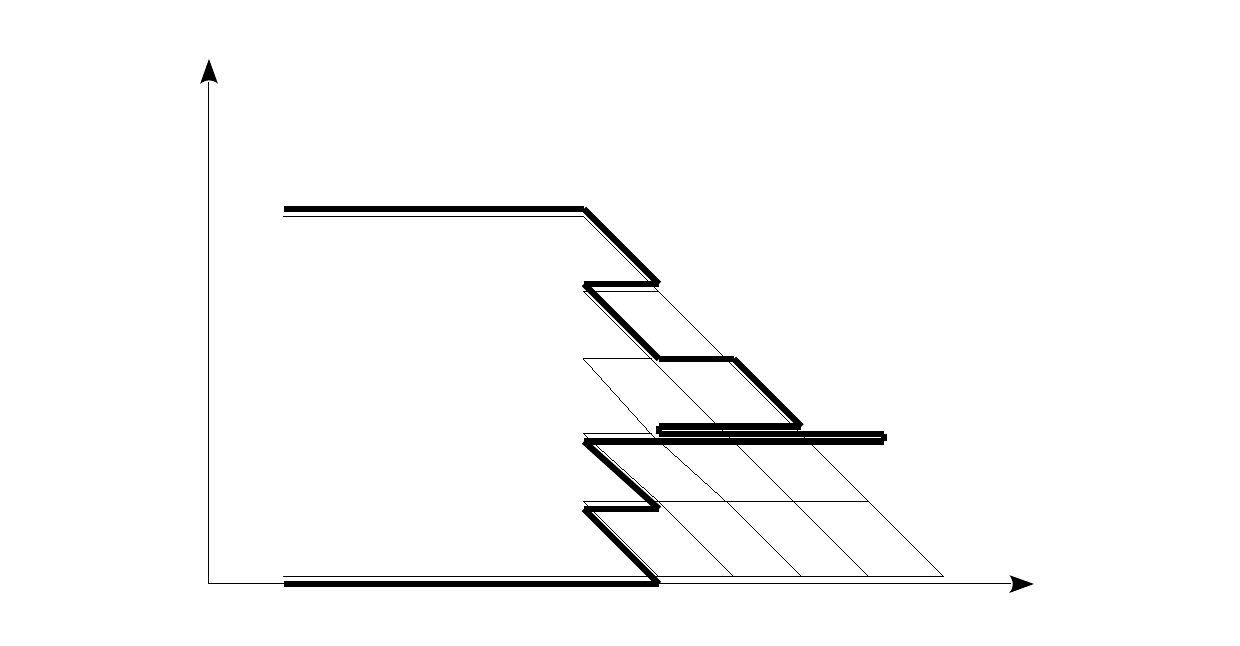, height=7cm}
\put(-10.8, 0.95){$(1, 0)$}
\put(-10.8, 4.2){$(1, 5)$}
\put(-6.3, 3.95){$(6, 4)$}
\put(-7.55, 4.95){$(5, 5)$}
\end{center} 
\caption{A walk from $(1,0)$ to $(1, 5)$ as a bold zigzag. It reaches height $5$ at the step $(6,4) \to (5,5)$. 
Then all the indicated zigzag walks (using only right and up-and-left steps)
 are such that they go from $(5,0) \to (5, 5)$ and have shortest length $10$. Their number is equal to $\De +1$. All of them intersect the bold walk at 
 $(6,4)$ as well.}
\label{simplestwalks}
\end{figure}

\begin{cor}\label{maincor}
There is a dominating set $\widetilde {\mathcal D}_{N, g}$
of the Morse functions with at most $N \geq 2g + 2$ critical points and one local minimum on a genus $g \geq 2$ surface 
 consisting of at most 
 $$
 M_{N-2} \left( 1 - \frac{\De}{(\De+1)^{1+\frac{1}{\De}}}  \right)
 $$
 Morse functions. This is true since 
 $W^{-1}(\mathcal D)$ is such a set if we take only one Morse function 
 in each $W^{-1}(w)$, where $w \in \mathcal D$, see Remark~\ref{kicsi_dom}.
\end{cor}

\end{document}